\documentclass[11pt]{amsart}
\usepackage{amsxtra}
\usepackage{amssymb}
\usepackage{mathrsfs} 
\usepackage{amsmath}
\usepackage{graphicx}
\usepackage{enumerate}
\usepackage{array,color}
\usepackage{hyperref}
\usepackage{comment}
\usepackage[all]{xy}
\theoremstyle{plain}

\newtheorem{thm}{Theorem}[section]
\newtheorem{lem}[thm]{Lemma}
\theoremstyle{definition}
\newtheorem{defn}[thm]{Definition}
\newtheorem{remk}[thm]{Remark}

\newcommand{\vocab}[1]{\textbf{#1}}

\def \wt {\widetilde}

\newcommand*\pFqskip{8mu}
\catcode`,\active
\newcommand*\pFq{\begingroup
        \catcode`\,\active
        \def ,{\mskip\pFqskip\relax}%
        \dopFq
}
\catcode`\,12
\def\dopFq#1#2#3#4#5{%
        {}_{#1}F_{#2}\biggl[\genfrac..{0pt}{}{#3}{#4};#5\biggr]%
        \endgroup
}

\newcommand*\pPqskip{8mu}
\catcode`,\active
\newcommand*\pPq{\begingroup
        \catcode`\,\active
        \def ,{\mskip\pPqskip\relax}%
        \dopPq
}
\catcode`\,12
\def\dopPq#1#2#3#4#5{%
        {}_{#1}\mathbb{P}_{#2}\biggl[\genfrac..{0pt}{}{#3}{#4};#5\biggr]%
        \endgroup
}

\makeatletter
\newcommand{\symprod}{\mathbin{\mathpalette\make@circled s}}
\newcommand{\make@circled}[2]{%
  \ooalign{$\m@th#1\smallbigcirc{#1}$\cr\hidewidth$\m@th#1#2$\hidewidth\cr}%
}
\newcommand{\smallbigcirc}[1]{%
  \vcenter{\hbox{\scalebox{0.77778}{$\m@th#1\bigcirc$}}}%
}
\makeatother

\begin{document}
\title[Pascal's Matrix, Point Counting on Elliptic Curves and Prolate Functions]
{Pascal's Matrix, Point Counting on Elliptic Curves and Prolate Spheroidal Functions}
\author[W.~Riley Casper]{W.~Riley Casper}
\address{
Department of Mathematics \\
California State University Fullerton \\
Fullerton, CA 92831\\
U.S.A.
}
\email{wcasper@fullerton.edu}  
\date{}
\keywords{Prolate spheroidal functions, discrete-discrete bispectrality, 
Pascal matrix, counting points on elliptic curves}
\subjclass[2020]{Primary 33C90; Secondary 11C20, 34L10, 11T06}
\begin{abstract}
The eigenvectors of the $(N+1)\times (N+1)$ symmetric Pascal matrix $T_N$ are analogs of prolate spheroidal wave functions in the discrete setting.
The generating functions of the eigenvectors of $T_N$ are prolate spheroidal functions in the sense that they are simultaneously eigenfunctions of a third-order differential operator and an integral operator over the critical line $\{z\in\mathbb{C}: \text{Re}(z) = 1/2\}$.
For even, positive integers $N$, we obtain an explicit formula for the generating function of an eigenvector of the symmetric pascal matrix with eigenvalue $1$.
When $N=p-1$ for an odd prime $p$, we show that the generating function is equivalent modulo $p$ to $(\# E_z(\mathbb F_p)-1)^2$, where $\# E_z(\mathbb F_p)$ is the number of points on the Legendre elliptic curve $y^2 = x(x-1)(x-z)$ over the finite field $\mathbb F_p$.
Furthermore when $N=p^n-1$, our generating function is the square of a period of $E_z$ modulo $p^n$ in the open $p$-adic unit disk.
\end{abstract}
\maketitle
\section{Introduction}
Integral operators which have the \emph{prolate spheroidal property} of commuting with a differential operator arise in random matrix theory and signal processing \cite{S,SP,TW1,TW2}.
The most famous example of this is the \vocab{Slepian differential operator}
$$\partial_x(x^2-\tau^2)\partial_x-\omega^2x^2$$
originating from signal processing which commutes with the time and band-limiting operator
$$(T_{\omega,\tau}f)(x) = \int_{-\tau}^\tau \frac{\sin(x-y)}{x-y}f(y)dy.$$
The joint eigenfunctions of these operators are called \vocab{prolate spheroidal wave functions} and have found important applications in numerical analysis, spectral theory, and geophysics.
Importantly, the kernel $K_\omega(x,y)$ of the time and band-limiting operator $T_{\omega,\tau}$ is given (up to a scalar multiple) by
$$K_\omega(x,y) = \int_{-\omega}^\omega \psi_{\text{exp}}(x,z)\overline{\psi_{\text{exp}}(y,z)}dz,$$
where here $\psi_{\text{exp}}$ is the \vocab{exponential bispectral function}
$$\psi_{\text{exp}}(x,z) = e^{2\pi i xz}.$$
Here, a function being \vocab{bispectral} means that $\psi(x,z)$ is a family of eigenfunctions for an operator in $x$, and \emph{simultaneously} a family of eigenfunctions for an operator in $z$ \cite{DG}.

When we replace the exponential bispectral function with either the Airy or Bessel bispectral functions, we obtain the integral operators with the prolate spheroidal property found by Tracy and Widom in random matrix theory \cite{TW1,TW2}.  In fact, this generic recipe can be shown to generate integral operators with the prolate spheroidal property for many bispectral functions \cite{BHY1,CY1}.  
The construction has been extended in many various directions, including discrete-continuous and matrix-valued bispectral functions (eg. orthogonal polynomials or orthogonal matrix polynomials satisfying differential equations) \cite{CGYZ4}.
The eigenfunctions of the commuting differential operator naturally generalize prolate spheroidal functions in various contexts.
Recently, Connes and Moscovici found interesting similarities between eigenvalues of Slepian's prolate operator and the zeros of the Riemann zeta function \cite{CM}.
One motivation of the present paper is to find similar number-theoretic connections for prolate operators in one of these extended contexts.

The $(N+1)\times (N+1)$ (symmetric) Pascal matrix
\begin{equation}T_N = \left[\begin{array}{cccc}
\binom{0+0}{0} & \binom{0+1}{0} & \dots & \binom{0+N}{0}\\
\binom{1+0}{1} & \binom{1+1}{1} & \dots & \binom{1+N}{1}\\
\vdots & \vdots & \ddots & \vdots\\
\binom{N+0}{N} & \binom{N+1}{N} & \dots & \binom{N+N}{N}\\
\end{array}\right]
\end{equation}
is a discrete analog of the integral operator construction discussed above.
In particular, one can view $T_N$ as discrete integral operator
$$(T_N\vec v)_k = \sum_{j=0}^N K_N(j,k)v_j,$$
associated with the discrete-discrete bispectral function $\psi(x,y) = \binom{x}{y}$,
whose kernel is defined by
$$K_N(j,k) = \sum_{\ell=0}^N\psi(j,\ell)\psi(k,\ell) = \binom{j+k}{j}.$$

Based on the prolate spheroidal property discussed in the continuous setting, we should not be surprised that $T_N$ commutes with a discrete analog of a differential operator.
Specifically, tridiagonal matrices may be viewed as discrete versions of second-order differential operators.
In \cite{CZ}, the authors prove that $T_N$ commutes with the $(N+1)\times(N+1)$ tridiagonal matrix
\begin{equation}\label{eqn:jacobi}
J_N = \left[\begin{array}{cccc}
b(0) & a(1) &  0  & \dots \\
a(1) & b(1) & a(2) & \dots \\
 0  & a(2) & b(2) & \dots \\
\vdots & \vdots & \vdots & \ddots\\
\end{array}\right],
\end{equation}
with entries
$$a(n) = (N+1)^2n-n^3,\quad\text{and}\quad b(n) = 2n^3+3n^2+2n-(N+1)^2n.$$
Following this analogy, the eigenvectors of $J_N$ (equiv. of $T_N$) should be a discrete analog of prolate spheroidal wave functions.
Therefore we anticipate that the eigenvectors of $T_N$ to have many applications.
For example, one can use them to generate orthogonal bases of the eigenspaces of the binomial transform \cite{CZ}.
In particular, this makes finding explicit expressions for the eigenvectors an interesting problem.

When $N$ is even, Pascal's matrix $T_N$ has $\lambda=1$ as a simple eigenvalue.
Our first main theorem gives an explicit generating function formula for a corresponding eigenvector.
\smallskip

\noindent
{\bf{Theorem A.}} {\em{
Let $N$ be even.  Then the $(N+1)\times (N+1)$ Pascal matrix $T_N$ has a unique eigenvector $\vec v = (v_k)_{k=0}^N$ with eigenvalue $\lambda=1$ (normalized with $v_N=1$) given by the generating function formula
\begin{equation}
\sum_{k=0}^N v_kz^{N-k} = \pFq{2}{1}{-N/2,N/2+1}{-N}{z}\cdot \pFq{2}{1}{-N/2,N/2+1}{-N}{\frac{z}{z-1}}(1-z)^{N/2}.
\end{equation}
}}
\medskip

For any vector $\vec v\in\mathbb{C}^{N+1}$, we call the expression
$$f(\vec v;z) = \sum_{k=0}^N v_kz^{N-k}$$
the \vocab{generating function} of the vector $\vec v$.
It turns out that if $\vec v$ is an eigenvector of $J_N$, then it is an eigenvector of $T_N$ and the generating function $f(\vec v; z)$ is a classical \vocab{prolate spheroidal function} in the sense that it is simultaneously an eigenfunction of an integral operator and a differential operator.
\smallskip

\noindent
{\bf{Theorem B.}} {\em{
Let $\vec v\in\mathbb{C}^{N+1}$ be an eigenvector of the matrix $J_N$ from Equation \eqref{eqn:jacobi} with eigenvalue $\mu$.
Then $\vec v$ is an eigenvector of $T_N$ for some eigenvalue $\lambda$ of $T_N$ and the generating function $f(\vec v; z)$ satisfies the integral equation
\begin{equation}
\frac{1}{2\pi i}\int_{\text{Re}(w)=\frac{1}{2}} \frac{1}{w^{N+1}(1-w)^{N+1}(1-z+zw)}f\left(\vec v; w\right) dw = \lambda f(\vec v;z)    
\end{equation}
and the third-order differential equation
\begin{align}
\mu y\nonumber
  &= z^2(1-z)^2y''' + 3z(1-z)((N-1)z-N)y''\\
  &+ N((2N-5)z^2+(2-5N)z+2N+1)y'\\\nonumber
  &+ N((2N+1)z+N^2+N+1)y\nonumber
\end{align}
with $y=f(\vec v;z)$.
}}
\medskip

Finally, we turn to the problem of interpreting the meaning of the entries of $\vec v$.
Motivated by Connes and Moscovici's recent result \cite{CM}, one might hope that our prolate functions could be linked to some version of a zeta in a well-chosen finite context.
Over a finite field $\mathbb F_p$, the Hasse-Weil zeta function of an algebraic curve is related to the number of points on the curve over algebraic extensions of $\mathbb F_p$.
Likewise, our third main theorem links our generating function expression to the number of points on an elliptic curve over the finite field $\mathbb F_p$.
\smallskip

\noindent
{\bf{Theorem C.}} {\em{
Let $N=p-1$ for an odd prime $p$.  Then the eigenvector $\vec v$ of the $p\times p$ Pascal matrix $T_N$ from Theorem A satisfies
\begin{equation}
f(\vec v; z) = \sum_{k=0}^N v_kz^{N-k} \equiv (\#E_z(\mathbb F_p)-1)^2\mod p,
\end{equation}
where here $\#E_z(\mathbb F_p)$ is the number of $\mathbb F_p$-points on the elliptic curve $E_z$ in the Legendre family of curves 
$$E_z: y^2 = x(x-1)(x-z).$$
}}
\medskip

This theorem follows from a Pfaffian transformation and the equality
$$\pFq{2}{1}{-N/2,N/2+1}{-N}{z} \equiv \pPq{2}{1}{\phi,\phi}{-}{z;p}\mod p$$
where $\pPq{2}{1}{\phi,\phi}{-}{z;p}$ is the period function 
$$\pPq{2}{1}{\phi,\phi}{-}{z;p} = \sum_{x\in \mathbb F_p} \phi(x(x-1)(x-z))$$
for $\phi(\cdot) = \left(\frac{\cdot}{p}\right)$ the Legendre symbol \cite{F}.

\subsection{A deeper $p$-adic picture}
The connection between prolate spheroidal functions in signal processing and number theory presented by Theorem C is surprising at first glance.
We can find a deeper explanation if we consider evaluating our generating function on $p$-adic numbers for an odd prime $p$.

Consider the series
\begin{equation}
F(z) = \pFq{2}{1}{1/2,1/2}{1}{z} = \sum_{k=0}^\infty \binom{2k}{k}^2\frac{z^k}{8^k}.
\end{equation}
For $z$ in the complex unit disk, this series is equal to a period of the Legendre elliptic curve $E_z$.
Moreover, the series converges $p$-adically for $z\in \mathbb Q_p$ with $|z|_p < 1$ and is a $p$-adic solution of the associated hypergeometric differential equation \cite{dwork,kedlaya}
\begin{equation}\label{eqn:p-adic diff}
z(1-z)f''(z) + (1-2z)f'(z) - \frac{1}{4}f(z) = 0.
\end{equation}

Let $N_n = p^n-1$. 
By comparing coefficients of the series, it is clear that the generating function of the eigenvector of the $(N_n+1)\times (N_n+1)$ symmetric Pascal matrix $T_{N_n}$ found in Theorem A, ie.
$$U_n(z) = \pFq{2}{1}{-N_n/2,N_n/2+1}{-N_n}{z}\cdot \pFq{2}{1}{-N_n/2,N_n/2+1}{-N_n}{\frac{z}{z-1}}(1-z)^{N_n/2}$$
satisfies
$$U_n(z) \equiv F(z)^2\mod p^n\quad\text{for all $z\in \mathbb Q_p$ with $|z|_p < 1$}.$$
Consequently we have a $p$-adic convergence
$$U_n(z) \rightarrow F(z)^2\quad\text{for all $z\in \mathbb Q_p$ with $|z|_p < 1$}.$$

There is a deeper conceptual reason this convergence occurs.
The generating function $U_n(z)$ is a solution of the third-order differential equation in Theorem B, with $\mu = (N^2+2N)/2$ and $N=N_n$.
In the limit as $n\rightarrow\infty$, this differential equation is the symmetric square of Equation \ref{eqn:p-adic diff} (see Definition \ref{defn: symm sq})
\begin{equation}\label{eqn:p-adic square}
z^2(1-z)^2f'''(z) + 3z(1-z)(1-2z)f''(z) + (1-7z(1-z))f'(z) - (1/2-z)f(z) = 0,
\end{equation}
whose solution space is spanned by products of solutions of Equation \ref{eqn:p-adic diff}.
Meanwhile, the integral operator in Theorem B converges to the Kummer transformation $T: f(z)\mapsto \frac{1}{1-z}f\left(\frac{1}{1-z}\right)$.
The transformation $T$ acts on solutions of Equation \ref{eqn:p-adic square} and is Frobenius-equivariant, so it must preserve the filtration of the solution space by Frobenius slopes.
In particular, the slope $0$ (ie. unit-root line) consists of eigenvectors of $T$.
Now since $U_n(z)$ is prolate, it must converge to a solution of Equation \ref{eqn:p-adic square} which is an eigenfunction of $T$, ie. something in the unit-root line.
Since the unit-root line of Equation \ref{eqn:p-adic square} of is spanned by $F(z)^2$ and $U_n(0) = F(0)^2 = 1$, we get that $U_n(z)$ must converge to $F(z)^2$.

\begin{remk}
The fact that $F(z)^2$ rather than $F(z)$ shows up here is natural, since in particular, it prevents the need for a square root in the Kummer transformation, allowing $T$ to have a local series representation $T_n$ near $z=0$.
\end{remk}

\subsection{A brief history of Pascal}
It is worth noting that the symmetric Pascal matrix $T_N$ has a long mathematical history.
According to Muir, F. Caldarera first considered $T_N$ and proved $\det(T_N)=1$ in 1871.  Rutishauser later proved that $T_N$ has a Cholesky decomposition in terms of the binomial transform \cite{newman}, giving a simpler proof of Calderara's theorem.
The behavior of the \emph{eigenvalues} of $T_N$ for $N=p^n-1$ modulo $p$ was studied by Strauss and Waterhouse in $1986-87$ \cite{strauss,waterhouse}, but no eigenvectors were found.  
In more modern works, properties of the Pascal matrix have been explored by Edelman and Strang \cite{Edelman} and Brawer and Pirovino \cite{brawer}, among others.

As far as we know, our paper is the first to obtain an explicit expression for any eigenvector of $T_N$.
In fact, \cite{brawer} was the first to point out that $T_N$ has a rational eigenvector with eigenvalue $1$, ie. that the diophantine system
$$\sum_{k=0}^N \binom{j+k}{k}v_k = v_j$$
has a nontrivial solution in $\mathbb{Q}$ (and hence $\mathbb{Z}$) when $N$ is even.  However, solutions to this system were given numerically only for $N=2,4,5,8$ and $10$, and until now no explicit formula for a solution was known.
Theorem A above provides the explicit solution
$$v_\ell = \sum_{\substack{j+k=\ell\\0\leq j,k\leq N/2}}\frac{(-N/2)_j(-N/2)_k(N/2+1)_j(-3N/2-1)_k}{j!k!(-N)_j(-N)_k},\quad 0\leq \ell\leq N,$$
where here $(q)_k = q(q+1)\dots(q+k-1)$ is the (rising) Pochhammer symbol.

\section{Generating functions of eigenvectors}
\subsection{Basic properties}
We start by proving some basic properties of the action of $T_N$ and $J_N$ on generating functions of vectors.
\begin{defn}
Let $\vec v = (v_k)_{k=0}^N\in\mathbb C^{N+1}.$
We define the \vocab{generating function} of $\vec v$ to be the polynomial
$$f(\vec v;z) = \sum_{k=0}^N v_kz^{N-k}.$$
\end{defn}

The property of a vector $\vec v$ being an eigenvector of $T_N$ translates directly to a certain functional equation on the generating function of $\vec v$ via the following lemma.

\begin{lem}\label{lem:T-action}
Let $\vec v = (v_k)_{k=0}^N\in\mathbb C^{N+1}.$
The generating function of $\vec v$ satisfies
$$f(T_N\vec v;z) = \left(\frac{z}{z-1}\right)^{N+1}z^Nf\left(\vec v; 1-\frac{1}{z}\right) + \frac{1}{z}\sum_{j=0}^Nv_j\binom{N+1+j}{j}\pFq{2}{1}{1,j+N+2}{N+2}{\frac{1}{z}}.$$
\end{lem}
\begin{proof}
From the binomial series
\begin{align*}
f(T_N\vec v; z)
  & = \sum_{k=0}^N \sum_{j=0}^N \binom{j+k}{j}v_jz^{N-k}\\
  & = \sum_{j=0}^N v_j \left(\sum_{k=0}^\infty \binom{j+k}{j}z^{N-k} - \sum_{k=N+1}^\infty \binom{j+k}{j}z^{N-k}\right)\\
  & = \sum_{j=0}^N v_j \left[z^N\left(1-\frac{1}{z}\right)^{-j-1} -  z^{-1}\binom{j+N+1}{j}\pFq{2}{1}{1,j+N+2}{N+2}{z^{-1}}\right]\\
  & = \frac{z^{2N+1}}{\left(z-1\right)^{N+1}}f(\vec v; 1-1/z) -  z^{-1}\sum_{j=0}^Nv_j\binom{j+N+1}{j}\pFq{2}{1}{1,j+N+2}{N+2}{z^{-1}}.
\end{align*}
\end{proof}

As an immediate consequence, we can reframe the search for eigenvectors of $T_N$ in terms of a certain residue integral eigenvalue problem.
\begin{thm}\label{thm:int eqn}
Let $\vec v\in\mathbb C^{N+1}$.  Then
$\vec v$ is an eigenvector of $T_N$ with eigenvalue $\lambda$ if and only if
$$\frac{1}{2\pi i}\int_{\text{Re}(w)=\frac{1}{2}} \frac{1}{w^{N+1}(1-w)^{N+1}(1-z+zw)}f\left(\vec v; w\right) dw = \lambda f(\vec v;z).$$
\end{thm}
\begin{proof}
Since $f(T_N\vec v; z)$ is a polynomial, the previous Lemma tells us it will be equal to the polynomial part of
$$\frac{z^{2N+1}}{(z-1)^{N+1}}f\left(\vec v;1-\frac{1}{z}\right).$$
Thus by Cauchy's residue theorem 
$$f(T_N\vec v;z)
   = \frac{1}{2\pi i(z-1)^{N+1}}\oint_{|u-1|=1} \left(\frac{1}{u-z} - \sum_{k=0}^{N}\frac{(z-1)^k}{(u-1)^{k+1}}\right)u^{2N+1}f\left(\vec v; 1-\frac{1}{u}\right) du.$$
Calculating the geometric sum and using the change of variables $w=1-1/u$, we get
\begin{align*}
f(T_N\vec v;z)
  & = \frac{1}{2\pi i}\oint_{|u-1|=1} \left(\frac{1}{u-1}\right)^{N+1}\frac{u^{2N+1}}{u-z}f\left(\vec v; 1-\frac{1}{u}\right) du\\
  & = \frac{1}{2\pi i}\int_{\text{Re}(w)=\frac{1}{2}} \frac{1}{w^{N+1}(1-w)^{N+1}(1-z+zw)}f\left(\vec v; w\right) dw.
\end{align*}
The statement of the theorem follows immediately.
\end{proof}

Likewise, the property of a vector $\vec v$ being an eigenvector of $J_N$ translates directly into a property of the generating function of $\vec v$.  This time, we get that $f(\vec v;z)$ is a polynomial eigenfunction of a certain third-order differential equation.

\begin{thm}\label{thm:diff eqn}
A vector $\vec v\in \mathbb C^{N+1}$ is an eigenvector of $J$ with eigenvalue $\mu$ if and only if $y = f(\vec v; z)$ is a solution of
\begin{align*}
\mu y
  &= z^2(1-z)^2y''' + 3z(1-z)((N-1)z-N)y''\\
  &+ N((2N-5)z^2+(2-5N)z+2N+1)y'\\
  &+ N((2N+1)z+N^2+N+1)y
\end{align*}
\end{thm}
\begin{proof}
Let $\vec v\in\mathbb C^{N+1}$ and let 
$$a(z) = (N+1)^2z-z^3,\quad\text{and}\quad b(z) = 2z^3+3z^2+2z-(N+1)^2z$$
be the polynomials definining the structure of the Jacobi matrix $J_N$ in Equation \eqref{eqn:jacobi}.
Then since $a(N+1)=0$ and $a(0)=0$,
\begin{align*}
f(J\vec v;z)
  & = \sum_{k=0}^N v_k (a(k+1)z^{N-k-1} + b(k)z^{N-k} + a(k)z^{N-k+1})\\
  & = \sum_{k=0}^N v_k (a(k+1)z^{N-k-1} + b(k)z^{N-k} + a(k)z^{N-k+1})\\
  & = \sum_{k=0}^N v_k (z^{-1}a(N-z\partial_z+1)z^{N-k} + b(N-z\partial_z)z^{N-k} + za(N-z\partial_z)z^{N-k})\\
  & = (z^{-1}a(N+1-z\partial_z)+ b(N-z\partial_z) + za(N-z\partial_z)) \cdot f(\vec v; z)
\end{align*}
The rest of the theorem follows from explicit calculation of the operator 
$$z^{-1}a(N+1-z\partial_z)+ b(N-z\partial_z) + za(N-z\partial_z).$$
\end{proof}

Combining the two previous theorems, the statement of Theorem B readily follows.
\begin{proof}[Proof of Theorem B]
Suppose that $\vec v$ is an eigenvector of $J_N$ with eigenvalue $\mu$.
Then since $J_N$ is a Jacobi matrix, it must have simple spectrum.
Since $J_N$ and $T_N$ commute, it follows that $\vec v$ is also an eigenvector of $T_N$ for some eigenvalue $\lambda$.
The statement of Theorem B then follows automatically from Theorem \ref{thm:diff eqn} and Theorem \ref{thm:int eqn}.
\end{proof}

\subsection{Eigenvectors and the binomial transform}
The symmetric pascal matrix $T_N$ has the Cholesky decomposition
$$T_N = B_NB_N^*,$$
where here $B_N$ is the $(N+1)\times (N+1)$ \vocab{binomial transform}
$$(B_N\vec v)_j = \sum_{k=0}^N(-1)^k\binom{j}{k}v_k,\quad 0\leq j\leq N.$$
The binomial transform is involutory and conjugates $T_N$ to $T_N^{-1}$.  Consequently $\lambda$ is an eigenvalue of $T_N$ if and only if $\lambda^{-1}$ is an eigenvalue of $T_N$.  Moreover, the binomial transform defines an isomorphism between the associated eigenspaces \cite{CZ}
$$\xymatrix{
E_\lambda(T_N) \ar@/^1pc/[r]^B & E_{1/\lambda}(T_N) \ar@/^1pc/[l]^B
}.$$
This symmetry of the eigendata translates to some properties of the corresponding generating functions.
This is made explicit in the next lemma.
\begin{lem}
Let $\vec v = (v_k)_{k=0}^N\in\mathbb C^{N+1}.$
The generating function of $\vec v$ satisfies
$$f(B_N^*\vec v;z) = (z-1)^Nf\left(\vec v; \frac{z}{z-1}\right)$$
$$f(B_N\vec v;z) = \left(\frac{z}{z-1}\right)^{N+1}f(\vec v; 1-z) - z^{-1}\sum_{j=0}^N v_j(-1)^j\binom{N+1}{j}\pFq{2}{1}{1,N+2}{N+2-j}{\frac{1}{z}}.$$
\end{lem}
\begin{proof}
From the binomial theorem,
\begin{align*}
f(B_N^*\vec v; z)
  & = \sum_{k=0}^N \sum_{j=0}^N (-1)^k\binom{j}{k}v_jz^{N-k}\\
  & = \sum_{j=0}^N v_j z^{N-j}\sum_{k=0}^N \binom{j}{k}(-1)^kz^{j-k}\\
  & = \sum_{j=0}^N v_j z^{N-j}(z-1)^j = (z-1)^Nf\left(\vec v;\frac{z}{z-1}\right).
\end{align*}
Also from binomial series,
\begin{align*}
f(B_N\vec v; z)
  & = \sum_{k=0}^N \sum_{j=0}^N (-1)^j\binom{k}{j}v_jz^{N-k}\\
  & = \sum_{j=0}^N v_j(-1)^j\left(\sum_{k=0}^\infty\binom{k}{j}z^{N-k} - (-1)^j\sum_{k=N+1}^\infty\binom{k}{j}z^{N-k} \right)\\
  & = \sum_{j=0}^N v_j(-1)^j\left(z^{N+1}\frac{1}{z-1}(z-1)^{-j} - \sum_{k=N+1}^\infty\binom{k}{j}z^{N-k} \right)\\
  & = \left(\frac{z}{z-1}\right)^{N+1}f(\vec v,1-z) - z^{-1}\sum_{j=0}^N v_j(-1)^j\binom{N+1}{j}\pFq{2}{1}{1,N+2}{N+2-j}{z^{-1}}.
\end{align*}
\end{proof}

Using the previous lemma, we can relate the generating function of an eigenvector $\vec v$ to the generating function of $B_N\vec v$.
\begin{thm}
Let $\vec v\in\mathbb C^{N+1}$ be an eigenvector of $T_N$ with eigenvalue $\lambda$.
Then
$$\lambda f(\vec v;z) = (z-1)^Nf\left(B_N\vec v; \frac{z}{z-1}\right).$$
\end{thm}
\begin{proof}
Using the previous lemma with $B_N\vec v$ in place of $\vec v$, we find
$$\lambda f(\vec v; z) = f(T_N\vec v; z) = f(B_N^*B_N\vec v;z) = (z-1)^Nf\left(B_N\vec v; \frac{z}{z-1}\right).$$
\end{proof}

The binomial transform also permutes the eigenvectors of $J_N$ \cite{CZ}, and specifically interchanges the eigenspaces with eigenvalue $\lambda$ and $N^2+2N-\lambda$
$$\xymatrix{
E_\lambda(J_N) \ar@/^1pc/[r]^B & E_{N^2+2N-\lambda}(T_N) \ar@/^1pc/[l]^B
}.$$

As a consequence, when $N$ is even $J_N$ has an eigenvector with eigenvalue $\frac{N^2+2N}{2}$.
This eigenvector is necessarily an eigenvector of $T_N$ with eigenvalue $1$.
\begin{thm}\label{thm:which one}
Let $N$ be even.
Then $(N^2+2N)/2$ is an eigenvalue of $J_N$ and $$E_{(N^2+2N)/2}(J_N)\subseteq E_1(T_N).$$
\end{thm}
\begin{proof}
Note that $N+1$ is odd and that $J_N$ has simple spectrum, so $J_N$ must have an odd number of nontrivial eigenspaces.
The binomial transform acts as an involution on the set of eigenspaces, so at least one eigenspace of $J_N$ must be preserved by $B_N$.
Since $B_N$ sends the eigenspace of $\lambda$ to the eigenspace of $N^2+2N-\lambda$, the only eigenvalue that is fixed is $\lambda=(N^2+2N)/2$.

Since $J_N$ has simple spectrum, the corresponding eigenspace is spanned by a single vector
$$E_{(N^2+2N)/2}(J_N) = \text{span}\{\vec v\}.$$
Also since $B_N$ sends this eigenspace to itself, we know $\vec v$ is an eigenvector of $B_N$.

Finally, since $J_N$ and $T_N$ commute, $\vec v$ is also an eigenvector of $T_N$.
Since $B_N\vec v\in\text{span}(\vec v)$, we know that $\vec v$ belongs to an eigenspace of $T_N$ that $B_N$ preserves.
The only possible candidate is the eigenspace for eigenvalue $1$.  The theorem follows immediately.
\end{proof}

\section{Explicit generating function formula and point couting}
\subsection{The eigenvector with eigenvalue $\lambda = 1$}

A generating function for an eigenvector with eigenvalue $1$ can be computed explicity.
The key to the computation comes from the fact that the generating function in this case exhibits some additional symmetry.
In terms of the differential operator, this symmetry can be described by the differential operator being the symmetric square of a second-order differential operator.
This allows us to solve the differential equation explicitly in terms of solutions of a second-order differential equation.

\begin{defn}\label{defn: symm sq}
Let $\{f_1(z),\dots, f_m(z)\}$ and $\{g_1(z),\dots, g_n(z)\}$ be bases of ther kernels of two monic differential operators $L$ and $\wt L$, of order $m$ and $n$, respectively.
The \vocab{symmetric product} of two monic differential operators $L\symprod \wt L$ is the unique monic differential operator whose kernel is spanned by $\{f_j(z)g_k(z): 1\leq j m,\ 1\leq k\leq n\}$.
If $L=\wt L$, then $L\symprod \wt L$ is called the \vocab{symmetric square} of $L$, and denote $L^{\symprod2}$.
\end{defn}

We first review a simple criteria for checking when a third-order differential operator is a symmetric square.
\begin{lem}[Singer \cite{Singer1} ]\label{lem:symcond}
A third-order differential operator
$$S = \partial_z^3 + u_2(z)\partial_z^2 + u_1(z)\partial_z + u_0(z)$$
is the symmetric square of the second-order differential operator
$$L = \partial_z^2 + v_1(z)\partial_z + v_0(z)$$
if and only if
\begin{align*}
u_2(z) &= 3v_1(z),\\
u_1(z) &= 4v_0(z) + v_1'(z) + 2v_1(z)^2,\\
u_0(z) &= 2v_0'(z) + 4v_0(z)v_1(z).
\end{align*}
\end{lem}

Using this criteria, we can prove that for the eigenvalue $1$ of $T_N$ (equivalently, the eigenvalue $(N^2+2N)/2$ of $J_N$) our differential operator is a symmetric square.
\begin{lem}\label{lem:symmsquare}
The third order differential equation
\begin{equation*}
S = \partial_z^3 - 3\frac{2z-1}{z(1-z)}\partial_z^2 + \left(\frac{N^2+2N-6}{z(1-z)} - \frac{N^2+2N}{z^2(1-z)^2}\right)\partial_z + \frac{N^2+2N}{z^2(1-z)} - \frac{\mu}{z^2(1-z)^2}
\end{equation*}
is the symmetric square of a second order differential operator if and only if $\mu = (N^2+2N)/2$.
In this case $S = L^{\symprod 2}$ for
$$L = \partial_z^2 + \frac{1-2z}{z(1-z)}\partial_z - \frac{(N^2+2N)(z^2-z+1) + 1}{4z^2(1-z)^2}.$$
\end{lem}
\begin{proof}
We need to determine when the overdetermined system of differential equations in Lemma \ref{lem:symcond} has a solution.
Solving the first two equations in the Lemma, we get
$$v_1(z) = \frac{1-2z}{z(1-z)}
\quad\text{and}\quad
v_0(z) = -\frac{(N^2+2N)(z^2-z+1) + 1}{4z^2(1-z)^2}.$$
Putting this into the third equation of the Lemma, we find $\mu = (N^2+2N)/2$.
The statement of the lemma follows immediately.
\end{proof}

The previous lemma allows us to build a fundamental set of solutions for our differential equation.
\begin{lem}\label{lem:gensoln}
The general solution of the third order differential equation 
$$y''' - 3\frac{2z-1}{z(1-z)}y'' + \left(\frac{N^2+2N-6}{z(1-z)} - \frac{N^2+2N}{z^2(1-z)^2}\right)y' + \frac{(N^2+2N)(1/2-z)}{z^2(1-z)^2}y = 0$$
has the basis of solutions
\begin{align*}
y_1 &= \left(\frac{z}{1-z}\right)^{N+1}
\pFq{2}{1}{-N/2,N/2+1}{-N}{1-z}^2\\
y_2 &= \pFq{2}{1}{-N/2,N/2+1}{-N}{1-z}\pFq{2}{1}{-N/2,N/2+1}{-N}{z}\\
y_3 &= \left(\frac{1-z}{z}\right)^{N+1}
\pFq{2}{1}{-N/2,N/2+1}{-N}{z}^2
\end{align*}
\end{lem}
\begin{proof}
If we do the change of variables $t = 2z-1$, then the differential equation
$$y'' + \frac{1-2z}{z(1-z)}y' - \frac{(N^2+2N)(z^2-z+1) + 1}{4z^2(1-z)^2}y = 0$$
becomes
$$4y'' - \frac{8t}{(1-t^2)}y' - \frac{(N^2+2N)(t^2+3) + 4}{(1-t^2)^2}y = 0,$$
which simplifies to the Legendre differential equation
$$(1-t^2)y'' - 2ty' + \left(\frac{N(N+2)}{4} - \frac{(N+1)^2}{1-t^2}\right)y = 0.$$
Two linearly independent solutions of this equation are given by the Ferrers function
$$
\left(\frac{1+t}{1-t}\right)^{N+1}\pFq{2}{1}{-N/2,N/2+1}{-N}{\frac{1}{2}-\frac{1}{2}t}
$$
and its $180$ degree rotation
$$
\left(\frac{1-t}{1+t}\right)^{N+1}\pFq{2}{1}{-N/2,N/2+1}{-N}{\frac{1}{2}+\frac{1}{2}t}.
$$
The statement of our lemma then follows by substituting $t = 2z-1$ and using Lemma \ref{lem:symmsquare}.
\end{proof}

Before proving Theorem A, we require one more identity regarding hypergeometric functions.
\begin{lem}\label{lem:helper}
For any nonnegative integer $N$
\begin{align*}
\pFq{2}{1}{-N/2,N/2+1}{-N}{\frac{z}{z-1}}(1-z)^{N/2}
  &= \pFq{2}{1}{-N/2,N/2+1}{-N}{z}(1-z)^{N+1}\\
  &+ (-1)^{N/2}\pFq{2}{1}{-N/2,N/2+1}{-N}{1-z}z^{N+1}
\end{align*}
\end{lem}
\begin{proof}
First notice that 
$(1-z)^{N+1}
\pFq{2}{1}{-N/2,N/2+1}{-N}{z}^2$
and
$\pFq{2}{1}{-N/2,N/2+1}{-N}{1-z}z^{N+1}$
are linearly independent solutions of the hypergeometric differential equation
$$z(1-z)y''+(-N+2Nz)y' - N(3N+1)y = 0$$
near $z=1$.
By the Pfaff transformation
$$\pFq{2}{1}{-N/2,N/2+1}{-N}{\frac{z}{z-1}}(1-z)^{N/2} = \pFq{2}{1}{-N/2,-3N/2-1}{-N}{z},$$
so that $F(z)$ is a solution of the same hypergeometric differential equation.
Therefore
\begin{align*}
\pFq{2}{1}{-N/2,N/2+1}{-N}{\frac{z}{z-1}}(1-z)^{N/2}
  &= A\cdot\pFq{2}{1}{-N/2,N/2+1}{-N}{z}(1-z)^{N+1}\\
  &+ B\cdot\pFq{2}{1}{-N/2,N/2+1}{-N}{1-z}z^{N+1}
    \end{align*}
for some constants $A$ and $B$.
Evaluating at $z=0$, we immediately see
$$A=\pFq{2}{1}{-N/2,N/2+1}{-N}{0} = 1.$$
To get $B$, we wish to take the limit as $z\rightarrow 1$.
Since $\pFq{2}{1}{-N/2,N/2+1}{-N}{z}$ is a palendromic polynomial, we have
$$\pFq{2}{1}{-N/2,N/2+1}{-N}{\frac{z}{z-1}}(1-z)^{N/2} = (-1)^{N/2}\pFq{2}{1}{-N/2,N/2+1}{-N}{\frac{z-1}{z}}z^{N/2}.$$
Thus by taking the limit, we find
$$B=(-1)^{N/2}\pFq{2}{1}{-N/2,N/2+1}{-N}{0} = (-1)^{N/2}.$$
This completes the proof.
\end{proof}

Combining all the lemmas above, we can now prove a theorem that is essentially the same as Theorem A.
\begin{thm}\label{thm:almostA}
For $N> 0$ an even integer and $\mu = (N^2+2N)/2$, the differential equation
\begin{align*}
 \mu\wt y
  &= z^2(1-z)^2\wt y''' + 3z(1-z)((N-1)z-N)\wt y''\\
  &+ N((2N-5)z^2+(2-5N)z+2N+1)\wt y'\\
  &+ N((2N+1)z+N^2+N+1)\wt y
\end{align*}
has the polynomial solution 
$$\wt y = \pFq{2}{1}{-N/2,N/2+1}{-N}{z}\pFq{2}{1}{-N/2,N/2+1}{-N}{\frac{z}{z-1}}(1-z)^{N/2}.$$
\end{thm}
\begin{proof}
If we do the substitution $\wt y=z^{N+1}y,$
then the differential equation simplifies to the third order equation in Lemma \eqref{lem:symmsquare}.
Therefore Lemma \eqref{lem:gensoln} tells us
\begin{align*}
\wt y 
  &= (1-z)^{N+1}
\pFq{2}{1}{-N/2,N/2+1}{-N}{z}^2\\
  &+ (-1)^{N/2}z^{N+1}\pFq{2}{1}{-N/2,N/2+1}{-N}{1-z}\pFq{2}{1}{-N/2,N/2+1}{-N}{z}
\end{align*}
is a solution.
Now if we apply the result of Lemma \ref{lem:helper}, the theorem follows immediately.
\end{proof}

We finish this section with a proof of Theorem A.
\begin{proof}[Proof of Theorem A]
The function 
$$f(z) = \pFq{2}{1}{-N/2,N/2+1}{-N}{z}\pFq{2}{1}{-N/2,N/2+1}{-N}{\frac{z}{z-1}}(1-z)^{N/2}$$
is a polynomial of degree $N$, and therefore $f(z) = f(\vec v;z)$ for some vector $\vec v\in\mathbb{C}^{N+1}$.
By Theorem \ref{thm:almostA} and Theorem \ref{thm:diff eqn}, the vector $\vec v$ is an eigenvector of $J_N$ with eigenvalue $\frac{N^2+N}{2}$.
Finally, by Theorem \ref{thm:which one} we know $\vec v$ is an eigenvector of $T_N$ with eigenvalue $1$.
\end{proof}

\subsection{Counting points over finite fields}
It turns out that the generating function of the eigenvector with eigenvalue $1$
$$f(z) = \pFq{2}{1}{-N/2,N/2+1}{-N}{z}\pFq{2}{1}{-N/2,N/2+1}{-N}{\frac{z}{z-1}}(1-z)^{N/2}$$
has a lot of symmetries modulo $p$.
In particular, one can check
$$f(z) = f(1-z) = f\left(\frac{1}{z}\right) = f\left(\frac{1}{1-z}\right) = f\left(1-\frac{1}{z}\right) = f\left(\frac{z-1}{z}\right),$$
for all $z\in\mathbb F_p$ with $z\neq 0,1$.
The set of M\"obius transformations
$$G = \left\lbrace z,1-z,\frac{1}{z},\frac{1}{1-z},1-\frac{1}{z},\frac{z-1}{z}\right\rbrace$$
defines a subgroup of $\text{PGL}_2(\mathbb{Z})$.
This group is strongly linked with the Legendre family of elliptic curves
$$E_z: y^2=x(x-1)(x-z)$$
over $\mathbb F_p$.
In particular, two curves $E_z$ and $E_w$ are isomorphic if and only if $w=\chi(z)$ for some $\chi\in G$.
Consequently, the value $f(z)$ should be some isomorphism invariant of the elliptic curve $E_z$.
In this section, we prove exactly that.
Namely, we prove Theorem C that
$$f(z)\equiv (\# E_z(\mathbb F_p)-1)^2\mod p.$$

One well-known result from number theory is that $\# E_z(\mathbb F_p)-1$ modulo $p$ is given (up to a sign) by the \vocab{Igusa polynomial}
$$H_p(z) = \sum_{k=0}^{(p-1)/2}\binom{(p-1)/2}{k}^2z^k.$$
In particular, this leads to the Deuring-Hasse criterion that $E_z$ is supersingular if and only if $H_p(z)\equiv 0\mod p$ \cite{Hasse,Silverman}.
\begin{thm}[Hasse \cite{Hasse}]
Let $p>2$ be prime and $z\in \mathbb F_p$ with $z\neq 0,1$.
Then
$$(-1)^{(p-1)/2}H_p(z)\equiv 1-\#E_z(\mathbb F_p)\mod p.$$
\end{thm}

\begin{lem}
Let $N=p-1$ for an odd prime $p$.
Then
$$\pFq{2}{1}{-N/2,N/2+1}{-N}{z} \equiv \pFq{2}{1}{-N/2,N/2+1}{-N}{z/(z-1)}(1-z)^{N/2}\mod p.$$
\end{lem}
\begin{proof}
Using the Pfaffian identity and that $N\equiv -1\mod p$ we have
\begin{align*}
\pFq{2}{1}{-N/2,N/2+1}{-N}{z/(z-1)}(1-z)^{N/2}
  & = \pFq{2}{1}{-N/2,-3N/2-1}{-N}{z}\\
  & = \sum_{k=0}^N \frac{(-N/2)_k(-3N/2-1)_k}{(-N)_k}\frac{z^k}{k!}\\
  &\equiv \sum_{k=0}^N \frac{(-N/2)_k(N/2+1)_k}{(-N)_k}\frac{z^k}{k!}\mod p\\
   & = \pFq{2}{1}{-N/2,N/2+1}{-N}{z}.
\end{align*}
\end{proof}

\begin{lem}
Let $N=p-1$ for an odd prime $p$.
Then
$$H_p(z) \equiv \pFq{2}{1}{-N/2,N/2+1}{-N}{z} \mod p.$$
\end{lem}
\begin{proof}
Since $(-N)_k\equiv k!\mod p$ and
\begin{align*}
\binom{N/2}{k}
  & = \frac{(N/2)(N/2-1)\dots(N/2-k+1)}{k!}\\
  & \equiv \frac{(-1/2)(-1/2-1)\dots(-1/2-k+1)}{k!}\mod p
   = (-1)^k\frac{(1/2)_k}{k!},
\end{align*}
we calculate
\begin{align*}
H_p(z)
 & \equiv \sum_{k=0}^{N/2}\frac{(1/2)_k(1/2)_k}{(-N)_k}\frac{z^k}{k!}\mod p \\
 & \equiv \pFq{2}{1}{-N/2,N/2+1}{-N}{z} \mod p.
\end{align*}
\end{proof}
The statement of Theorem C is simply the combination of Hasse's Theorem and the previous two lemmas.

\section*{Acknowledgements}
\thanks{The research of W.R.C. has been supported by an AMS-Simons Research Enhancement Grant.}

\end{document}